\documentclass{amsart} 
\DeclareSymbolFontAlphabet{\mathcal}{symbols}

\usepackage{soul} 

\usepackage[T1]{fontenc}
\usepackage{tikz}
\usetikzlibrary{fadings}
\usepackage[a4paper]{geometry}
\usepackage{framed}
\usepackage{color}
\usepackage{amstext}
\usepackage{amsthm}
\usepackage{amssymb}
\usepackage{varioref}
\usepackage{amssymb}
\usepackage{mathrsfs}
\usepackage{enumerate}
\usepackage[all]{xy}
\usepackage{graphicx}
\usepackage{url}
\usepackage{float}
\usepackage[shortlabels]{enumitem}
\usepackage[pagebackref=true]{hyperref}
\usepackage{hyperref}
\usepackage{backref}
\hypersetup{
    colorlinks,
    citecolor=blue,
    filecolor=black,
    linkcolor=blue,
    urlcolor=black
}

\usepackage{float}
\makeatletter
\@namedef{subjclassname@2020}{%
  \textup{2020} Mathematics Subject Classification}
\makeatother
\title[Truncations of stratified spaces]{Homotopy truncations of homotopically stratified spaces}

\date{\today}

\author{David Chataur}
\address{Lamfa\\
Universit\'e de Picardie Jules Verne\\
33, rue Saint-Leu\\
80039 Amiens Cedex~1\\
         France}
\email{David.Chataur@u-picardie.fr}

\author{Martintxo Saralegi-Aranguren}
\address{Laboratoire de Math{\'e}matiques de Lens\\  
      EA 2462 \\
      Universit\'e d'Artois\\
         SP18, rue Jean Souvraz\\
          62307 Lens Cedex\\
         France}
\email{martin.saraleguiaranguren@univ-artois.fr}

\author{Daniel Tanr\'e}
\address{D\'epartement de Math{\'e}matiques\\
         UMR-CNRS 8524 \\
         Universit\'e de Lille\\
         59655 Villeneuve d'Ascq Cedex\\
         France}
\email{Daniel.Tanre@univ-lille.fr}

\thanks{The first author was supported by the research project ANR-18-CE93-0002  ``OCHOTO'' . 
The third author was partially supported by 
by the Proyecto PID2020-114474GB-100
and 
the ANR-11-LABX-0007-01  ``CEMPI''}



\makeatletter
\renewcommand\l@subsection{\@tocline{2}{0pt}{2pc}{5pc}{}}
\renewcommand\l@subsubsection{\@tocline{3}{0pt}{4pc}{10pc}{}}
\makeatother

\dedicatory{Dedicat a la mem\'oria d'Agust\'{\i} Roig, matem\`atic, catal\`a i amic de sempre.}
\theoremstyle{plain}
\newtheorem{theorem}{Theorem}

\newtheorem{proposition}{Proposition}[section]
\newtheorem{theoremb}[proposition]{Theorem}

\newtheorem{corollary}[proposition]{Corollary}

\theoremstyle{definition}
\newtheorem{definition}[proposition]{Definition}
\newtheorem{example}[proposition]{Example}

\theoremstyle{remark}
\newtheorem{remark}[proposition]{Remark}

\numberwithin{equation}{section}
\newcommand{\secref}[1]{Section~\ref{#1}}
\newcommand{\subsecref}[1]{Subsection~\ref{#1}}
\newcommand{\thmref}[1]{Theorem~\ref{#1}}
\newcommand{\propref}[1]{Proposition~\ref{#1}}

\newcommand{\remref}[1]{Remark~\ref{#1}}

\newcommand{\defref}[1]{Definition~\ref{#1}}

\def\R{{\mathbb R}}

\def\ov{\overline}

\def\cE{{\mathcal E}}

\def\cL{{\mathcal L}}

\def\cM{{\mathcal M}}
\def\cN{{\mathcal N}}

\def\cT{{\mathcal T}}

\def\cS{{\mathcal S}}

\def\crG{{\mathscr G}}

\def\1{{\mathbf 1}}

\def\tc{{\mathtt c}}

\def\tv{{\mathtt v}}

\def\tts{{\mathtt s}}

\def\ttD{{\mathtt D}}

\def\ttS{{\mathtt S}}

\def\R{\mathbb{R}}
\def\Z{\mathbb{Z}}

\def\Map{{\rm Map}}

\def\id{{\rm id}}

\def\codim{{\rm codim\,}}

\def\sing{{\rm Sing}}

\def\tF{{\widetilde{F}}}

\def\ev{{\tt eval}}

\def\rc{{\mathring{\tc}}}

\def\menos{\backslash}

\newcounter{ejemplo}

\newcounter{figura}

\def\top{{\mathbf{Top}}}

\def\sset{{\mathbf{Sset}}}

\def\cyl{{\mathrm{cyl}}}
\def\rcyl{{\mathring{{\mathrm{cyl}}\;}}}

\def\holink{{\mathrm{holink}}}
\def\holinks{{\mathrm{holink}_{\tts}}}

\begin{document} 
\begin{abstract} 
Intersection homology of Goresky and MacPherson can be defined from the Deligne sheaf,
obtained from truncations of  complexes of sheaves.
As intersection homology is not the homology of a particular space, 
 the search for a family of spaces whose homologies have properties analogous to intersection homology 
has developed. For some stratified spaces, M. Banagl has introduced such a family by using a topological truncation:
the original link is replaced by a truncation of its homological Moore resolution. 

In this work, we study the dual approach in the Eckmann-Hilton sense : 
we consider the stratified space obtained by replacing the original link by a Postnikov approximation. 
The main result is that our construction  restores the space constructed 
by Gajer to establish an intersection  Dold-Thom theorem. 
 We are conducting this study within the general framework of Quinn's homotopically stratified spaces.
\end{abstract}


\maketitle

\section*{Introduction}
In \cite{GM1}, M. Goresky and R. MacPherson re-establish Poincar\'e duality  with rational coefficients
for some stratified spaces, called pseudomanifolds. For that, they introduce perversity functions, $\ov{p}$,
and  intersection homology defined from the Deligne sheaf, 
obtained from a succession of ad'hoc truncations of complexes of sheaves (\cite{GM2}).
This process is done at an algebraic level and not at a topological one; 
intersection homology is not the homology of a space.

\medskip
M. Banagl has developed (\cite{MR2662593}) the idea of the construction of a family of spaces, $I_{\ov{p}}X$, 
from a stratified pseudomanifold $X$ and whose homologies
 have properties analogous to intersection homology. In particular, if $D\ov{p}$ is the complementary perversity of $\ov{p}$,
a Poincar\'e duality, similar to that of  \cite{GM1,GM2}, 
is required between the (co)homologies of $I_{\ov{p}}X$ and $I_{D\ov{p}}X$.
 In \cite{MR2928934,MR3544285}, Banagl achieves this goal for 2-strata pseudomanifolds with certain requirements,
 as the existence of a flat link bundle.
 The construction of $I_{\ov{p}}X$ uses a topological truncation:
 the original link is replaced by a truncation of its homological Moore resolution. 
 If the pseudomanifold has only isolated singularities, let us also mention the work of M. Spiegel
(\cite{Spiegel})
 where a truncation of the link is performed 
 with respect to any homology theory given by a connective ring spectrum.
 In the case of the ordinary Eilenberg-MacLane spectrum, Spiegel 
recovers exactly the intersection homology.

\medskip
Here, we develop an Eckmann-Hilton dual version of Banagl's construction, $I_{\ov{p}}X$, replacing 
the truncation of the Moore  homological decomposition by 
a truncation of the Postnikov tower of the links. 
Using homotopical tools, the process fits well with the  homotopically stratified spaces introduced by F. Quinn
and we place this work in this setting, see \defref{def:quinnspace}.

\medskip
A second type of space also appears. To present it, let us  recall that intersection homology 
 is based on a choice of singular chains. This choice is made according to a perversity $\ov{p}$ and the
 chosen chains are called $p$-allowable. 
The tricky point is that the boundary of a $\ov{p}$-allowable chain is not necessarily $\ov{p}$-allowable. 
Thus, for having a chain complex, it is necessary to require the allowability of the chain and of  its boundary.
Here we do not work with chain complexes but with topological spaces or simplicial sets.
Therefore, the allowability property must be required on the simplexes and on all their faces. 
This construction has been introduced by P. Gajer in \cite{MR1404919,MR1489215}. 
From a stratified space $X$ and a perversity $\ov{p}$, Gajer gets a simplicial set $\crG_{\ov{p}}X$.

\medskip
Let's detail our main result for a Quinn's homotopically stratified space $X$ of  depth 1. 
Denote by $S$ the singular stratum of $X$
and by $\holinks(X,S)$ the stratified homotopy holink introduced by Quinn and formed of paths beginning
in $S$ which never return to $S$.  
The space $X$ is the homotopy pushout of 
$$ \xymatrix{
 X\menos S&\holinks(X,S)\ar[l]_-{\ev_{1}}\ar[r]^-{\ev_{0}}&S,
 }
 $$
where $\ev_{0}$ and $\ev_{1}$ are evaluation maps of paths. 
Let $\ov{p}$ be a perversity on $X$,
 of complementary perversity $D\ov{p}$.
As recalled in \secref{sec:truncationlink}, a Postnikov stage $P_{\ell}Y$ of a topological space $Y$ 
is a particular case of a localization. The process of localization  also exists for maps and is called fibrewise
localization. 
We denote it $ \widetilde{P}_{\ell}$.
Our main result states as follows for spaces of depth~1.

 \begin{theorem}\label{prop:gajerneighbor}
 Let $(X,\ov{p})$ be a perverse manifold homotopically stratified  space with connected links and
   two strata, $S$ and $X\menos S$.
 Then there is a  homotopy equivalence 
between the realisation of $\crG_{\ov{p}}X$ and the 
homotopy pushout of $\ev_{1}$ with the
fibrewise Postnikov $D\ov{p}(S)$-localization of $\ev_{0}\colon \holinks(X,S)\to S$.
  \end{theorem}

  For instance, if $X=\rc Y$ is the cone of apex $\tv$ on a topological space $Y$, 
  the previous construction gives  the $D\ov{p}(\tv)$-stage of the Postnikov tower of $Y$
  which is known to be a Gajer space.

\medskip
In the general case, we need a linkwise localization, obtained by induction on the depth  and 
described in \subsecref{sec:local}.
The main result is stated in \thmref{thm:gajer2strata} and, as announced, means
that the Eckmann-Hilton dual of Banagl's construction gives the Gajer space.

\medskip
\paragraph{\bf Outline of the paper}
Sections~\ref{sec:pseudo} and \ref{sec:quinn} are basic recalls on stratified objects and tools: 
Quinn's homotopicaly stratified spaces, holink, perversity,...
In \secref{sec:gajer}, we present the Gajer spaces  
(\cite{MR1404919}) already studied in \cite{CST9,CST10} and complete their properties
concerning fibrations and homotopy pushouts.
In \secref{sec:truncationlink}, we specify the concept of linkwise localization and prove the main theorem. An example is also given.

\medskip
\paragraph{\bf Notation and convention}
Let $\sset$ be the category of simplicial sets and $\top$ be the category of 
weak-Hausdorff compactly generated spaces (\cite{MR251719}).
We denote by $\sing\colon \top\to\sset$ the functor given by the singular chains
and by $|-|\colon \sset\to \top$ the realisation functor.
We use the  notation $\Delta[n]$, $\partial \Delta[n]$, $\Lambda[n,j]$  in the simplicial case
and
$\Delta^n$, $\partial \Delta^n$,  $\Lambda^n_{j}$ for their associated polyhedron. 
A manifold is supposed to be a connected, separable metric space. 
All path spaces are given the compact-open topology. 

\medskip
We are grateful to the referee for 
her/his comments and suggestions which helped us to improve 
the manuscript.

%
\section{Perversity on stratified spaces}\label{sec:pseudo}
Among the first structures adapted to singularities, there is the complex of manifolds of Whitney (\cite{MR19306}):
an amalgamation of manifolds of different dimensions. 
This concept  led  to the notions of Whitney and Thom-Mather stratified spaces.
In their pioneer works (\cite{GM1,GM2}), Goresky and MacPherson use \emph{pseudomanifolds}.
Here we are mainly concerned with the homotopically stratified spaces of Quinn (\cite{Qui}), 
recalled in \secref{sec:quinn}.

\medskip
 We denote by $\top$ the category of weak-Hausdorff compactly generated spaces (henceforth called ``space'')
with morphisms the continuous maps (henceforth called ``map''), 
see   \cite[Section 2]{MR251719}. 
This category verifies the conditions required by Hirschhorn in \cite[Section 1.1.1]{MR1944041}.
Let us now enter in the stratified world. 

\begin{definition}\label{def:filtered}
A \emph{filtered space} is a space, $X$, endowed with a filtration
$$
X^{-1}=\emptyset\subseteq X^{0} \subseteq X^{1} \subseteq 
\dots \subseteq X^{n-2} \subseteq X^{n-1} \subseteq X^{n} =X,
$$
by closed subsets. 
Such a filtration gives a partition of $X$ by path-connected, locally closed subspaces
defined as the non-empty path-components of $X_{i}=X^{i}\menos X^{i-1}$
and called \emph{strata}.
The set of strata is denoted by $\cS_{X}$ (or  $\cS$ if there is no ambiguity).
A \emph{stratified space} is a filtered space whose set of strata satisfies the \emph{Frontier condition:} 
$$S_{i}\cap \ov{S_{j}}\neq \emptyset\quad \text{implies}\quad S_{i}\subset \ov{S_{j}}.$$
\end{definition}

Each stratum of a filtered space has  a \emph{formal dimension} given by
$\dim S=i$ if  $S\subset X^{i}$. As well, the formal codimension is $\codim S=n-i$.
In a stratified space, the set $\cS$ of strata is a poset for the relation $S_{i}\preceq S_{j}$ if $S_{i}\subset \ov{S_{j}}$.
The maximal elements of $\cS$ are called \emph{regular} and we say \emph{bottom stratum} for a minimal one.

\begin{definition}\label{def:startifiedmap}
A \emph{stratified map} $f\colon X\to Y$ between two stratified spaces is a continuous map such that
for each stratum $S$ of $X$ there is a stratum $S^f$ of $Y$ verifying $f(S)\subset S^f$.
\end{definition}

 Let $X$ be a stratified space. A map $F\colon Z\times A\to X$ 
is \emph{stratum-preserving along $A$} if $F(\{z\}\times A)$ lies in a single stratum of $X$, for any $z\in Z$. 
If $A=I$, we say that $F$ is a  \emph{stratum-preserving homotopy} 
and denote by $\sim$ the associated equivalence relation.
A map $F\colon Z\times I\to X$ whose restriction to $Z\times [0,1[$ is stratum-preserving along $[0,1[$ is called
a \emph{nearly stratum-preserving homotopy.}

\begin{definition}\label{def:homotopyequistrat}
Two stratified spaces, $X$ and $Y$, are \emph{stratified homotopy equivalent} 
if there exist stratified maps $f\colon X\to Y$ and $g\colon Y\to X$
and stratum-preserving homotopies:
$f\circ g\sim \id_{Y}$ and $g\circ f\sim \id_{X}$.
We denote this relation  $X\simeq_{s} Y$.
\end{definition}

Let us  recall  a notion of fibration adapted to stratified spaces, see \cite[Definition 5.1]{MR1686706}.

\begin{definition}\label{def:stratfibration}
Let $X$ and $Y$  be stratified spaces. 
A   map $f\colon X\to Y$   is a \emph{stratified fibration} provided given any space $Z$ and any commuting diagram,
\begin{equation}\label{equa:lifting}\xymatrix{
Z\ar[rr]^-g\ar[d]_{i_{0}}
&& X\ar[d]^f\\
Z\times [0,1]\ar[rr]^{F}\ar@{-->}[rru]^-{\tF}
&&Y,
}
\end{equation}
with $F$ a stratum-preserving homotopy, there exists
a stratum-preserving homotopy, $\tF$, such that $f\circ \tF=F$ and $\tF(z,0)=g(z)$ for each $z\in Z$.
\end{definition}

The  mapping cylinder of a map
$f\colon X\to Y$ is endowed with the \emph{teardrop topology.} 
As set, the  mapping cylinder of   $f$   is the quotient      
$\cyl \,f= (X\times [0,1])\sqcup Y/\sim$ for the relation $(x,1)\sim f(x)$ for $x\in X$.
The teardrop topology on $\cyl \,f$ is defined as  the minimal topology such that:
\begin{enumerate}
\item  the inclusion $X\times [0,1[\to \cyl f$ is an open embedding, 
\item the map $c\colon \cyl \,f\to Y\times [0,1]$, defined by
$c(x,t)=(f(x),t)$ if $(x,t)\in X\times [0,1[$ and
$c(y)=(y,1)$ if $y\in Y$, is continuous.
\end{enumerate}
If  $f$ is a proper map between locally compact Hausdorff spaces,  the teardrop topology is the usual quotient  topology.
We send the reader to \cite{MR1695351,MR1763954}, \cite[Page 138-139]{MR1410261}
 for  basic properties of teardrop topology.

\medskip
Let $f\colon X\to Y$ be a map (not necessarily stratified) between stratified spaces.
 The \emph{strata of the mapping cylinder} are the strata of $Y$ and the products
$S\times [0,1[$ where $S$ is a stratum of $X$.
The \emph{open mapping cylinder} is the subspace $\mathring{\cyl} \,f=\cyl \,f \menos (X\times \{0\}$. 
Homotopy pushouts of two maps with the same domain being double mapping cylinders, they are also equipped with teardrop topology.

\medskip
Intersection homology of Goresky and MacPherson is defined from a parameter, called perversity. 

\begin{definition}\label{def:perversitegen}
A \emph{perversity on a filtered space, $X$,} is a map $\ov{p}\colon \cS_{X}\to \ov{\Z}={\Z}\cup \{\pm\infty\}$
taking the value 0 on the regular strata. The pair $(X,\ov{p})$
is called a \emph{perverse space.}
The \emph{top perversity}  $\ov t$ is
 defined by $\ov{t}(S)=\codim S-2$, for a singular stratum $S$. 
 Given a perversity $\ov p$ on $X$,  the \emph{complementary perversity} on $X$, $D\ov{p}$, is characterized 
 by $D\ov{p}(S)+\ov{p}(S)=\ov{t}(S)$, for any singular stratum $S$.
\end{definition}

\begin{definition}\label{def:backperversity}
Let $f\colon X\to Y$ be a stratified map and $\ov{p}$ be a perversity on $Y$. The \emph{pullback perversity of $\ov{p}$ by $f$}
is the perversity $f^*\ov{p}$ on $X$ defined on any singular stratum $S$ of $X$ 
by $f^*\ov{p}(S)=\ov{p}(S^f)$, 
where $S^f$ is the stratum of $Y$ containing $f(S)$.
In the case of the canonical injection of an open subset endowed with the induced filtration, $\iota\colon U\to Y$, 
we still denote by $\ov{p}$ the perversity $\iota^*\ov{p}$ and call it the \emph{induced perversity.}
\end{definition}

A perversity $\ov{p}$ allows a selection among the singular simplexes of a filtered space.

\begin{definition}\label{def:homotopygeom}
Let $(X,\ov{p})$  be a  perverse  space.
A simplex $\sigma\colon \Delta\to X$ is  
\emph{$\ov{p}$-allowable}  
if, for each singular stratum $S$, the set $\sigma^{-1}S$ verifies
\begin{equation}\label{equa:admissibleST}
\dim \sigma^{-1}S\leq \dim\Delta-\codim S +\ov{p}(S)=\dim\Delta-2-D\ov{p}(S),
\end{equation}
with the convention $\dim\emptyset=-\infty$.
\end{definition}

For this definition to have meaning, we need to specify the notion of dimension for the subspace 
$\sigma^{-1}S$ of a Euclidean  simplex.
There is some flexibility  and as King wrote in \cite{King}, any reasonable notion of dimension  gives the original 
intersection homology of \cite{GM1}.  
Here, we choose the following one, introduced by Gajer (\cite{MR1404919}) and 
revisited in \cite{CST9}.
 
\begin{definition}\label{def:dimension}
A subspace $A\subset \Delta$ of a Euclidean simplex is of \emph{polyhedral dimension} less than or equal to $\ell$
 if  $A$ is included in a polyhedron $Q$ with $\dim Q\leq \ell$. 
\end{definition}

As a face of a $\ov{p}$-allowable simplex is not necessarily $\ov{p}$-allowable, 
for having a simplicial set, we will strengthen the notion of $\ov{p}$-allowability  in \defref{def:gajersimplicialset}.

\section{Quinn's homotopically stratified spaces}\label{sec:quinn}

In this section, we recall the singular spaces introduced by F. Quinn in \cite{MR873296,Qui}. 
Allowing a study of stratified spaces with homotopical tools, there is
an extensive literature on their properties and applications.

\medskip
Let $M$ be a differentiable manifold.
In \cite{MR71081}, 
Nash defines a topological space $P_{\cN}(M)$ as the set of continuous map,
$\alpha\colon [0,1]\to M$ such that $\alpha(t)\neq \alpha(0)$ for all $t>0$, endowed with the compact-open topology.
Using a Riemanian metric, Nash proves that the tangent bundle of $M$ is a fibre deformation retract of $P_{\cN}(M)$. 
This result gives an alternative proof to the following theorem  of Thom (\cite{MR0054960}): 
the fibre homotopy type of the tangent bundle of  $M$ depends only on the topology of $M$, 
which infers the topological invariance of the Stiefel-Whitney classes.

\smallskip
In \cite{MR179795}, Fadell extends Nash's definition to obtain a topological analogue of the normal bundle 
of a locally flat $n$-dimensional topological manifold $S$ 
in a $(n+k)$-dimensional topological manifold $M^{n+k}$. 
Recall that $S$ is locally flat in $M$ if  each point of $S$ admits an open neighborhood $U$ in $S$
and there is an open embedding $h\colon U\times \R^k\to M$ such that $h(x,0)=x$ for all $x\in U$.
In analogy with $P_{\cN}(M)$, Fadell defines the space $\cE^0$ as the space of paths in $M$ 
which start on $S$ and never return in $S$, 
$$\cE^0=\{\omega\colon[0,1]\to M\mid \omega(t)\in S \text{ if, and only if, } t=0\},
$$
and the space $\cE$ which is the union of $\cE^0$ with the constant paths in $S$. 
Fadell proves that the evaluation in $t=0$, 
$\ev_{0}\colon (\cE,\cE^0)\to S$, gives a pair of locally trivial fibre spaces, thus of Hurewicz fibrations 
since the base $S$ is required to be paracompact in \cite{MR179795}. 
Moreover, the fibres of this pair are homotopy equivalent to $(\R^k,\R^{k}\menos 0)$.
Fadell also shows that the Whitney sum of $(\cE,\cE^0)$ with the Nash bundle
of $S$ is the restriction to $S$ of the Nash bundle of $M$.
This naturally leads to call  the pair $(\cE,\cE^0)$ the normal fibre space for  the inclusion $S\subset M$.  

\smallskip
In the case of Whitney spaces (\cite{MR19306}),
for each stratum $S$ there is a bundle over $S$ whose fibre is the link. 
This bundle comes from the existence of a tubular neighborhood (\cite{MR239613}).
In \cite{Qui}, Quinn develops a family of stratified spaces for which the total space can be recovered from a succession
of ``homotopical amalgamations of fibrations''.  
Most of the upcoming recalls are in the  work of Quinn (\cite{MR873296,Qui}). 
The first point is the stratified version of the Fadell normal bundle.

\begin{definition}\label{def:holink}
Let $X$ be a stratified space and $Y\subset X$. 
The \emph{homotopy link} (or \emph{holink}) of $Y$ in $X$ is the  space
$$\holink(X,Y)=\left\{\omega\colon [0,1]\to X\mid \omega(0)\in Y
\text{ and } \omega(t)\in X\menos Y \text{ for }
t\in ]0,1]
\right\}.$$
The \emph{stratified homotopy link} is a subspace of $\holink(X,Y)$ whose elements lie in a single stratum 
after leaving $Y$, i.e.,
$$\holinks(X,Y)=\left\{\omega\in \holink(X,Y)\mid \text{for some } S_{i}\in \cS_{X}, \;\omega(]0,1])\subset S_{i}
\right\}.$$ 
The evaluation at $0$ defines  maps $\ev_{0}\colon \holink(X,Y)\to Y$
and $\ev_{0}\colon \holinks(X,Y)\to Y$.
The stratified homotopy link is naturally filtered by
$$\holinks(X,Y)^{j}=\left\{\omega\in\holinks(X,Y)\mid \omega(1)\in X^{j}\right\}.$$
Let $S\in \cS_{X}$. The \emph{local holink} of $x_{0}\in S$, $\holinks(X,x_{0})$, is the fibre at $x_{0}$ of the map
$\ev_{0}\colon \holinks(X,S)\to S$ .
\end{definition}

\begin{definition}\label{def:tame}
{(\cite[Definitions 3.1 and 3.4]{MR1686706})}
The subspace $Y$ of a  space $X$ is  \emph{forward tame} in $X$ if there is a 
neighborhood $N$ of $Y$ in $X$ and a homotopy,
$h\colon N\times I\to X$, 
such that $h(-,0)$ is the inclusion $N\hookrightarrow X$,
the restriction $h(-,t)\colon Y\to X$ is the inclusion $Y\hookrightarrow X$ for each $t\in I$,
$h(N,1)=Y$ and $h((N\menos Y)\times [0,1[)\subset X\menos Y$.

The subspace $Y$ of a  stratified space $X$ is  \emph{stratified forward tame} in $X$ if there is a 
neighborhood $N$ of $Y$ in $X$ and a homotopy,
$h\colon N\times I\to X$, 
such that $h(-,0)$ is the inclusion $N\hookrightarrow X$,
the restriction $h(-,t)\colon Y\to X$ is the inclusion $Y\hookrightarrow X$ for each $t\in I$,
$h(N,1)=Y$ and $h((N\menos Y)\times [0,1[)\subset X\menos Y$
is stratum-preserving along $[0,1[$.
The map $h$ is called a \emph{nearly stratum-preserving strong deformation retraction} of $N$ to $X$.
\end{definition}

\begin{definition}\label{def:quinnspace}{(\cite{Qui})}
A stratified metric space $X$ is a \emph{homotopically stratified space} 
if the following two conditions are satisfied for every pair of strata with $S\preceq S'$,
\begin{enumerate}[i)]
\item $S$ is forward tame in $S\cup S'$,
\item  the evaluation map $\ev_{0}\colon \holink(S\cup S',S)\to S$ is a  fibration. 
\end{enumerate}
If, moreover, each stratum is a manifold without boundary and
is locally-closed in $X$, we say that $X$ is a 
\emph{manifold homotopically stratified space.}
\end{definition}

Whitney spaces  are typical examples of homotopically stratified spaces.  
Quinn also proves that a filtered space, with locally contractible skeleta  and  with conical neighborhoods
 (up to a stratified homotopy equivalence), is a homotopically stratified space. 
Strata and  holink spaces suffice to determine the stratified homotopy type of a 
homotopically stratified space (\cite[Lemma 2.4]{Qui})
and  to detect stratified homotopy equivalences.

\medskip
We list properties of homotopically stratified spaces, already present in the work of Quinn
(\cite{Qui}).

\begin{definition}\label{def:pure}
A subspace $Y$ of a stratified space is said \emph{pure} if it is closed and a union of strata.
\end{definition}

\begin{proposition}{\rm (\cite[Theorem 6.3 and Corollary 6.2]{MR1686706})}\label{prop:hughesneighbor}
Let $X$ be a homotopically stratified  space with a finite number of strata
and $Y\subset X$ be a pure subspace.
Then $Y$ is stratified forward tame in $X$ and the map
$\ev_{0}\colon \holinks(X,Y) \to  Y$ is a stratified fibration.
\end{proposition}

If a nearly stratum-preserving deformation retraction  of a neighborhood $N$ of a pure subset $Y$ exists
as in the previous proposition, the neighborhood $N$ can be replaced by the cylinder of an evaluation map.
 In the proof given in \cite[Lemma 2.4]{Qui}, some evaluation maps  turn out not to be continuous.
 Friedman has fixed it in  \cite[Appendix]{MR2009092}. 
We quote the particular case that we need.

\begin{proposition}{\rm (\cite[Proposition A.1]{MR2009092})}\label{prop:BruceGreg}
Let $X$ be a manifold homotopically stratified  space and $S$ be a bottom stratum.
Given a nearly stratum-preserving deformation retraction 
$h\colon N\times I \to N$ of a neighborhood $N$ of $S$, to $S$,
then $N$ is stratum-preserving homotopy equivalent to the mapping cylinder 
of the map $\ev_{0}\colon \holinks(N,S)\to S$.
\end{proposition}

\begin{remark}{(\cite[Page 49]{MR2354985})}\label{rem:pairscylinders}
With the notation of \propref{prop:BruceGreg}, denote  $M$ the mapping cylinder of the map
$\ev_{0}\colon \holinks(N,S)\to S$ and  $\cM$ the mapping cylinder of
$\ev_{0}\colon \holinks(X,S)\to S$. 
Recall that $\holinks(X,S)\simeq_{s}\holinks(N,S)$ 
since they are stratified homotopy equivalent to the stratified holink of ``small paths''.
There also exist stratum-preserving homotopy equivalences of pairs
$$
(N,N\menos S)\simeq_{s} (M,M\menos S) \simeq_{s} (M,\holinks(N,S))\simeq_{s} (\cM,\holinks(X,S))\simeq_{s}
(\cM,\cM\menos S).
$$
\end{remark}

\begin{corollary}
Let $X$ be a manifold homotopically stratified  space and $S$ be a bottom stratum.
Then $X$ is the homotopy pushout of
$$
\xymatrix{
S&\holinks(X,S)\ar[l]_-{\ev_{0}}\ar[r]^-{\ev_{1}}&X\menos S.
}$$
\end{corollary}

\begin{proof}
This is a consequence of the existence of a nearly stratum-preserving deformation retract neighborhood,
$N$,  of  $S$ in $X$, 
(\cite[Theorem 7.1]{MR1954237}), 
of \propref{prop:BruceGreg} and of a stratified homotopy equivalence 
$\holink_{s}(N,S)\simeq_{s}\holink_{s}(X,S)$.
\end{proof}

\section{Gajer simplicial set}\label{sec:gajer}

The notion of $\ov{p}$-allowable simplexes (\defref{def:homotopygeom}) is used  
by Gajer (\cite{MR1404919}) for the definition of a simplicial set as follows.

\begin{definition}\label{def:gajersimplicialset}
 Let $(X,\ov{p})$ be a  perverse space.
 A simplex $\sigma\colon \Delta^{\ell}\to X$ is  \emph{$\ov{p}$-full}  if
$\sigma$ and all its faces 
 are $\ov{p}$-allowable.
 \end{definition}

The set of $\ov{p}$-full simplexes is a simplicial set
verifying the Kan condition (\cite[Page 946]{MR1404919} or \cite[Proposition 2.3]{CST10}). 
We  denote it by $\crG_{\ov{p}}X$ and call it
 the \emph{Gajer simplicial set} associated to $(X,\ov{p})$.
In \cite{CST10}, we define the $\ov{p}$-intersection homotopy groups of a perverse space $(X,\ov{p})$
as the homotopy groups of $\crG_{\ov{p}}X$ and prove  a Hurewicz theorem linking these groups to the 
$\ov{p}$-intersection homology groups. 
We also show that they verify a Van Kampen theorem relatively to the open covers of $X$ and prove 
their topological invariance if the regular parts of $X$ and of its intrinsic stratification coincide. 
Let us  recall and establish some other results on  the spaces $\crG_{\ov{p}}X$.

 \smallskip
 Any stratified map $f \colon (X,\ov{p}) \to (Y,\ov{q})$ between perverse  spaces
 such that
$f^*D\ov{q}\leq D\ov{p}$ induces a simplicial map 
$\crG_{\ov{p},\ov{q}}f\colon \crG_{\ov{p}}X\to \crG_{\ov{q}}Y$.
Moreover, if  $\varphi\colon (X\times [0,1],\ov{p})\to (Y,\ov{q})$ is a  stratified homotopy between two stratified maps
$f,\,g\colon (X,\ov{p})\to (Y,\ov{q})$ with $f^*D\ov{q}\leq D\ov{p}$, then we also have $g^*D\ov{q}\leq D\ov{p}$
and the simplicial maps, $\crG_{\ov{p},\ov{q}}f$ and $\crG_{\ov{p},\ov{q}}g$ are homotopic
(\cite[Proposition 2.5]{CST10}).

\begin{proposition}{\rm (\cite[Corollary 2.6]{CST10})}\label{prop:gajerequiv}
 Let $f\colon (X,\ov{p})\to (Y,\ov{q})$ be a stratified homotopy equivalence between perverse stratified spaces
 such that $f^*D\ov{q}= D\ov{p}$.
 Then, the assignment $\sigma\mapsto f\circ \sigma$ induces a homotopy equivalence between
 $\crG_{\ov{p}}X$ and $\crG_{\ov{q}}Y$.
  \end{proposition}

 \begin{proposition}\label{prop:gajerfib}
 Let $f\colon E\to B$ be a stratified fibration such that $B$ has only one stratum
 and let $\ov{p}$ be a perversity on $E$.
 The fibre $F$ of $f$ is endowed with the induced filtration. Then, the map
 $\crG_{\ov{p}}f\colon \crG_{\ov{p}}E\to \sing\,B$ is a Kan fibration of fibre $\crG_{\ov{p}}F$.
  \end{proposition}
 
 \begin{proof}
 By definition of a Kan fibration, we have to solve the lifting problem
 $$\xymatrix{
 \land[\ell,k]\ar[r]\ar[d]&\crG_{\ov{p}}E\ar[d]\\
 \Delta[\ell]\ar[r]\ar@{-->}[ru]&\sing\,B.
 }$$
 By adjunction between $\sing$ and the realisation functor $|-|$, as  $\crG_{\ov{p}}E\subset \sing \,E$, this is equivalent to the lifting problem
 $$\xymatrix{
| \land[\ell,k]|\ar[r]^-{\tau}\ar@{^(->}[d]&E\ar[d]^{f}\\
| \Delta[\ell]|\ar[r]^-{\sigma}\ar@{-->}[ru]^-{\widetilde{\sigma}}&B
 }$$
 with $\widetilde{\sigma}$ of full $\ov{p}$-intersection.
 The  map $| \land[\ell,k]|\to | \Delta[\ell]|$ is homeomorphic to $\Delta^{\ell-1}\to \Delta^{\ell-1}\times [0,1]$.
Thus,  by hypothesis and \defref{def:stratfibration},
we have a stratum-preserving homotopy 
$\widetilde{\sigma}\colon \Delta^{\ell-1}\times [0,1]\to E$.
In this particular case, this means $f\circ \widetilde{\sigma}=\sigma$, $\widetilde{\sigma}(z,0)=\tau(z)$
and,  for any $z\in \Delta^{\ell-1}$, 
the image $\widetilde{\sigma}(z\times [0,1])$ is included in one stratum of $E$.
Let $S\subset E_{n-k}\menos E_{n-k-1}$ be a stratum. The previous properties imply
$\widetilde{\sigma}^{-1}(S)\cong \tau^{-1}(S)\times [0,1]
$ and
$\dim \widetilde{\sigma}^{-1}(S)=\dim \tau^{-1}(S)+1\leq \ell-k+\ov{p}(S)+1$. 
The same argument works for any face, thus $\widetilde{\sigma}\in \crG_{\ov{p}}E$.

\medskip 
By definition (\cite[Page 947]{MR1404919}) the map $f$ is a filtered fibration in the sense of Gajer.
From \cite[Theorem 2.2]{MR1404919}, we get a long exact sequence in homotopy
$\dots\to \pi_{*}(\crG_{\ov{p}}F)\to \pi_{*}(\crG_{\ov{p}}E)\to \pi_{*}(\sing\,B)\to\dots$.
If $K$ is the fibre of $\crG_{\ov{p}}E\to \sing\,B$, there is a simplicial map
$\crG_{\ov{p}}F\to K$. 
From the long exact sequence in homotopy of the fibration $\crG_{\ov{p}}f$ and the five lemma, we get  isomorphisms
$\pi_{*}(\crG_{\ov{p}}F)\cong \pi_{*}(K)$.
As $K$ and $\crG_{\ov{p}}F$ are Kan complexes, they are homotopy equivalent. 
 \end{proof}

 \begin{proposition}\label{prop:gajerpush}
Let $(X,\ov{p})$ be a perverse stratified space and $U$, $V$ be two open subsets of $X$,
endowed with the induced stratifications and perversities. We suppose that $U$, $V$, $U\cap V$
are path-connected.
Then  the following diagram is a homotopy pushout,
 \begin{equation}\label{equa:gajerpush}
 \xymatrix{
 \crG_{\ov{p}}(U\cap V)\ar[r]\ar[d]&\crG_{\ov{p}}U\ar[d]\\
 \crG_{\ov{p}}V\ar[r]&\crG_{\ov{p}}(U\cup V).
 }
 \end{equation}
 \end{proposition}
 
 \begin{proof}
 Denote by $K$ the pushout in $\sset$ of
 $\xymatrix{
 \crG_{\ov{p}}V&
 \crG_{\ov{p}}(U\cap V)\ar[l]\ar[r]&
 \crG_{\ov{p}}V.
 }$
 The two maps being injective, this is a homotopy pushout. By universal property, there is a canonical map
 $\varphi\colon K\to \crG_{\ov{p}}(U\cup V)$. By (\cite[Theorem 2.13]{CST10}),
 the map $\varphi$ induces an isomorphism in homology
 for any local coefficients and by the stratified Van-Kampen theorem  (\cite[Theorem 4.1]{CST10}), it induces an isomorphism between the
 fundamental groups. Therefore, the map $\varphi$ is a weak homotopy equivalence.
 \end{proof}
 
 \begin{corollary}\label{cor:holinkpush}
 Let $(X,\ov{p})$ be a homotopically stratified space with path-connected links and 
 $Y\subset X$ be a pure subset such that 
 $Y$, $X\menos Y$ and $\holink_{s}(X,Y)$ are path-connected. 
If $X$ is the homotopy pushout of
\begin{equation}\label{equa:losabia}
 \xymatrix{
Y&\holinks(X,Y)\ar[l]_-{\ev_{0}}\ar[r]^-{\ev_{1}}&X\menos Y,
  }\end{equation}
  then $\crG_{\ov{p}}X$ is the homotopy pushout of
  \begin{equation}\label{equa:losabiaG}
\xymatrix{
\crG_{\ov{p}}Y&&\crG_{\ov{p}}(\holinks(X,Y))\ar[ll]_-{\crG_{\ov{p}}\ev_{0}}\ar[rr]^-{\crG_{\ov{p}}\ev_{1}}&&
\crG_{\ov{p}}(X\menos Y).
 }\end{equation}
 \end{corollary}
 
 \begin{proof}
The homotopy pushout of \eqref{equa:losabia} is the space
$$M=X\menos Y\sqcup \holinks(X,Y)\times I\sqcup Y /\sim,$$
where the relation $\sim$ is generated by
$(\omega,0)\sim \omega(0)$
and
$(\omega,1)\sim \omega(1)$. 
A basis of open subsets of the teardrop topology of $M$ (\cite{MR2354985}) consists of
the open subsets of the product $\holinks(X,Y)\times ]0,1[$ 
with the product topology and of the sets
$(\ev_{0}^{-1}(W)\times ]0,\varepsilon[)\cup W$, $(\ev_{1}^{-1}(W')\times ]1-\varepsilon,1[)\cup W'$,
 where $W$ is an open subset of $Y$ and
 $W'$  an open subset of $X\menos Y$.
We consider the two open subsets
$$\left\{
\begin{array}{ccl}
U&=&(\holinks(X,Y)\times ]0,3/4[)\cup Y,\\[.2cm]
V&=&(\holinks(X,Y)\times ]1/4,1[)\cup (X\menos Y).
\end{array}\right.$$
Their union is $M$, their intersection is $\holinks(X, U)\times ]1/4,3/4[$ and there are homotopy equivalences
 $U\simeq Y$, $V\simeq X\menos Y$, $U\cap V=\holinks(X,Y)$.
We thus have a pushout with open subsets
$$
\xymatrix{
U\cap V\ar[r]\ar[d]&U\ar[d]\\
V\ar[r]&M.
}$$
The result follows from \propref{prop:gajerpush}.
 \end{proof}
 
\section{Postnikov truncation of links}\label{sec:truncationlink}

\subsection{Linkwise localization}\label{sec:local}
 We denote by $\cT$ one of the categories $\top$ or $\sset$, pointed or not.
The following recalls concern the localization along a map $f\colon A\to B$ of $\cT$ between cofibrant spaces,
see \cite{MR1944041}.

\medskip
A fibrant space $W$ is said \emph{$f$-local}  if the induced map of simplicial sets,
$f^*\colon \Map(B,W)\to \Map(A,W)$ is a weak equivalence.
 A map $g\colon X\to Y$ between cofibrant spaces is an \emph{$f$-local equivalence}
if the induced map of simplicial sets, $g^*\colon \Map(Y,W)\to \Map(X,W)$ is a weak equivalence for every $f$-local 
space $W$.
 An \emph{$f$-localization} of a space $X$ is an $f$-local space $\ov{X}$ with an $f$-local equivalence
$j_{X}\colon X\to \ov{X}$.
Let $ f \colon A \to B$ be an injection if $f\in\sset$ or an inclusion of cell complexes if $f\in\top$.
Then for every space $X$, there exists (\cite[Theorem 1.3.11]{MR1944041}) a natural $f$-localization  $j_{X}\colon X\to L_{f}X$ with $j_{X}$ a cofibration.

\medskip
Localization can also be done for maps as shown by the following proposition extracted from \cite[Theorem 6.1.3]{MR1944041}.

 \begin{proposition}\label{prop:fibrelocal}
 There is a functorial factorization of every map $p\colon X\to Z$   of $\cT$  
as $X\xrightarrow{i} \ov{L}_{f}X\xrightarrow{q} Z$, called \emph{fibrewise $f$-localization} of $p$,
 such that the following properties are satisfied.
 \begin{enumerate}[(1)]
\item The map $q$ is a fibration with $f$-local fibres and the map $i$ is a cofibration and an $f$-local equivalence.
 Moreover, for any $z\in Z$, the map  induced by~$i$ between the homotopy fibres is an $f$-localization.
 \item For any decomposition of $p$ as  $X\xrightarrow{j}W\xrightarrow{r} Z$, 
 where $r$ a fibration with $f$-local fibres, there exists $k\colon \ov{L}_{f}X\to W$
 such that $k\circ i=j$ and $r\circ k=q$.
 Moreover, if $j$ is another fibrewise $f$-localization, then $k$ is a weak equivalence.
 \end{enumerate}
 \end{proposition}
 
 Recall that a functorial factorization in $\cT$ means that any map in $\cT$ factors into a composite of two maps,
 in a way that depends functorially on commutative square (\cite{nlab:functorial_factorization}).
 
 \medskip
 Let $X$ be a homotopically stratified space  with a finite poset of strata $\cS$. If $S$ is a bottom stratum,
 the stratified fibration, $\ev_{0}\colon \holinks(X,S)\to S$, is a fibration. 
 Its fibre in $x_{0}\in S$ is the local holink $L=\holinks(X,x_{0})$ and we can apply
 to $\ev_{0}$ the fibrewise localization along any map $f\colon A\to B$ of $\top$.
 We then replace the space X, obtained as the homotopy pushout of 
 \begin{equation}\label{equa:leX}
 \xymatrix{
 X\menos S&\holinks(X,S)\ar[l]\ar[r]&S,
 }
 \end{equation}
 by the space $\cL_{f}X$, defined as the homotopy pushout of 
  \begin{equation}\label{equa:leLfX}
  \xymatrix{
 X\menos S&\holinks(X,S)\ar[l]\ar[r]&\ov{L}_{f}\holinks(X,S).
 }
 \end{equation}
 For instance, in the particular case of a cone on a space $ X=\rc Y$,
 of apex $\tv$, stratified by  $\tv\preceq Y\times [0,1[$, we have
  \begin{equation}\label{equa:coneY}
 \xymatrix{
 Y\times ]0,1[&\holinks(\rc Y,\tv)\ar[l]_-{\ev_{1}}\ar[r]^-{\ev_{0}}&\tv.
 }
 \end{equation}
 The map $\ev_{1}$ is a weak equivalence. The fibration $\ev_{0}$ can be trivialized and its fibre has the
 homotopy type of $Y$. We therefore get $\cL_{f}(\rc Y)\simeq L_{f}Y$.
 In short, the link $Y$ has been localized along $f$.
 If the stratified space has more than two strata, we  repeat the previous process.  
 
 \medskip
 We  can also choose different localizing maps for each stratum: let
 $\Phi$ be a correspondence that associates to each stratum $S$ a map $\Phi(S)\colon A_{S}\to B_{S}$ of $\top$.
 The \emph{linkwise $\Phi$-localization,} $\cL_{\Phi}X$, of $X$ is obtained by induction on the number of strata by
 starting with $\cL_{\Phi}X=X$ if $X$ has only regular strata.
 In the general case, let $S$ be a bottom stratum of $X$. 
 The image by $\cL_{\Phi}$ of $X\menos S$ and $\holinks(X,S)$ is defined by induction
 since they have one stratum less than $X$.
 We therefore have a map
 $\cL_{\Phi}\holinks(X,S) \to \cL_{\Phi}(X\menos S)$.
 Let us also consider 
  the fibrewise $\Phi(S)$-localization of $\cL_{\Phi}(\ev_{0})\colon \cL_{\Phi}\holinks(X,S)\to S$:
 $$\cL_{\Phi}\holinks(X,S)\xrightarrow{i_{S}}  \ov{L}_{\Phi(S)}(\cL_{\Phi}\holinks(X,S))\xrightarrow{q_{S}} S.$$
  We define $\cL_{\Phi}X$ as the homotopy pushout of
 \begin{equation}\label{equa:linkwise}
 \xymatrix{
 \ov{L}_{\Phi(S)}(\cL_{\Phi}\holinks(X,S))&\cL_{\Phi}\holinks(X,S)\ar[l]\ar[r]&\cL_{\Phi}(X\menos S).
 }\end{equation}
We study it in the case of the Postnikov localization.

\subsection{Postnikov truncation}\label{subsec:postnikov}
Let $\ell\geq 0$ and $f_{\ell}\colon \ttS^{\ell+1}\to \ttD^{\ell+2}$ 
be the standard inclusion in $\top$ of the $(\ell+1)$-sphere in the $(\ell+2)$-ball,
see \cite[Section 1.5]{MR1944041}. 
A space $X$ is $f_{\ell}$-local if, and only if, $\pi_{i}X= 0$ 
for $i>\ell$ and every choice of basepoint.
The Postnikov projection $\rho_{\ell}\colon X\to P_{\ell}X$ is an $f_{\ell}$-localization map.
We also call it the \emph{Postnikov $\ell$-localization} of $X$.
A map $g\colon X\to Y$ is an $f_{\ell}$-local equivalence if, and only if,
$g$ induces isomorphisms $g_{*}\colon \pi_{i}X\to \pi_{i}Y$
for $i>\ell$ and every choice of basepoint.

 \medskip
The main result of this section shows the Gajer simplicial set  as a linkwise localization.

\begin{theoremb}\label{thm:gajer2strata}
Let $(X,\ov{p})$ be a perverse manifold homotopically stratified  space with a finite number of strata and connected 
local holinks.
Let $\Phi$ be the correspondence associating to any stratum $S\in \cS$ the injection 
$\ttS^{D\ov{p}(S)+1}\to \ttD^{D\ov{p}(S)+2}$.
Then, the linkwise localization $\cL_{\Phi}X$ has the homotopy type of the realisation
$|\crG_{\ov{p}}X|$ of the Gajer simplicial set associated to $(X,\ov{p})$.
\end{theoremb}

%
\begin{proof}[Proof of \thmref{thm:gajer2strata}]
The result is obvious  for spaces with only  one stratum. 
We use an induction on the number of strata and  assume the result  true for 
the spaces having strictly less than $k$ strata.
 Let $(X,\ov{p})$ be  as in the statement with $k$ strata. We
choose a bottom stratum $S$ of $X$. We know from \propref{prop:hughesneighbor} that
the evaluation map $\ev_{0}$ is a stratified fibration and,  
thus a fibration since $S$ is unstratified.
From \propref{prop:hughesneighbor}, we also know that $S$ admits a neighborhood $N$ in $X$ 
with a nearly stratum-preserving retraction taking $N$ into $S$ relatively to~$S$. 
Moreover, from \propref{prop:BruceGreg}, we have the existence of a stratified homotopy equivalence 
between $N$ and  the mapping cylinder of $\ev_{0}\colon \holinks(N,S)\to S$,
$$N \simeq_{s} \rcyl(\ev_{0})=
 (\holinks(N,S)\times [0,1[)\sqcup S/(\omega,0)\sim \omega(0).
 $$
 This mapping cylinder  has the teardrop topology and a  filtration defined by $S$ and the ends of the paths in 
 $\holinks(N,S)$. 
 The map $\ov{\ev}_{0}\colon \rcyl(\ev_{0})\to S$ sends $s\in S$ to itself and $(\omega,t)$ to $\omega(0)$. This is a stratified fibration of fibre $\rc L$, 
see \cite[Proposition 3.3]{MR2130860}.
Recall that $P_{\ell}$ and $\widetilde{P}_{\ell}$ are, respectively, the Postnikov $\ell$-truncation and
the fibrewise Postnikov $\ell$-truncation.
We consider the following diagram in $\sset$:
 \begin{equation}\label{equa:bigpush}
 \xymatrix{
 P_{D\ov{p}(S)}(\crG_{\ov{p}}L)\ar[d]&
 \crG_{\ov{p}}L\ar[l]\ar[r]\ar[d]&
 \crG_{\ov{p}}\rc L\ar[d]\\
  \widetilde{P}_{D\ov{p}(S)}(\crG_{\ov{p}}\holinks(N,S))\ar[d]_{q}&
\crG_{\ov{p}}\holinks(N,S)\ar[l]_-{i}\ar[r]^-{j}\ar[d]_{\crG_{\ov{p}}(\ev_{0})}&
\crG_{\ov{p}}\rcyl(\ev_{0})\ar[d]_{\crG_{\ov{p}}(\ov{\ev}_{0})}\\
 \sing \,S&\ar@{=}[l] \sing\,S\ar@{=}[r]&\sing\, S.
 }
 \end{equation}
 Here, $L=\holinks(N,x_{0})$ is the local  holink of $x_{0}\in S$.
 The left column  is the fibrewise  Postnikov $D\ov{p}$(S)-localization of $\crG_{\ov{p}}({\ev}_{0})$
 and the right one is the image of the fibration $\ov{\ev}_{0}$ by $\crG_{\ov{p}}$.
As  $\crG_{\ov{p}}\rc L$ is homotopy equivalent to the
 Postnikov $P_{D\ov{p}(S)}$-section of~$\crG_{\ov{p}}L$ (\cite[Corollary 3.7]{CST10}),
 we deduce from the point (2) of \propref{prop:fibrelocal}, the existence of a homotopy equivalence
 $k\colon \widetilde{P}_{D\ov{p}(S)}(\crG_{\ov{p}}\holinks(N,S))\to \crG_{\ov{p}}\rcyl(\ev_{0})$
 such that $\crG_{\ov{p}}(\ov{\ev}_{0})\circ k=q$ and $k\circ i=j$.

\medskip
 Finally \propref{prop:gajerequiv} implies the existence of a homotopy equivalence between
$\crG_{\ov{p}}\rcyl(\ev_{0})$ and  $\crG_{\ov{p}}N$ and we have obtained a homotopy equivalence 
 $ \widetilde{P}_{D\ov{p}(S)}(\crG_{\ov{p}}\holinks(N,S))\simeq \crG_{\ov{p}}N$.
 In the following diagram of simplicial sets,  
 \begin{equation}\label{equa:firstpush}
 \xymatrix{
 \crG_{\ov{p}}\holinks(N,S)\ar[r]^-{\simeq}\ar[d]
 &\crG_{\ov{p}}(\rcyl(\ev_{0})\menos S)\ar[d]\ar[r]^-{\simeq}
 &\crG_{\ov{p}}(N\menos S)\ar[d]\\
\widetilde{P}_{D\ov{p}(S)}(\crG_{\ov{p}}\holinks(N,S))\ar[r]^-{\simeq}
&\crG_{\ov{p}}(\rcyl(\ev_{0}))\ar[r]^-{\simeq}
&\crG_{\ov{p}}N,
 }
 \end{equation}
 the  horizontal maps are homotopy equivalences. 
 For the top ones, this property comes from \remref{rem:pairscylinders} and for the lower ones 
 this has been established above. We conclude that the diagram \eqref{equa:firstpush} is  a homotopy pushout.
 
 \medskip
 We consider  the open covering $X=N\cup (X\menos S)$.
 As $S$ and $L$ are path-connected and $\ev_{0}$ is a fibration, we deduce that  $\holinks(X,S)$ is 
 path-connected and so is $N$.
 We can apply \propref{prop:gajerpush} and get a homotopy pushout
\begin{equation}\label{equa:secondpush}
\xymatrix{
 \crG_{\ov{p}}(N\menos S)\ar[r]\ar[d]&\crG_{\ov{p}}(X\menos S)\ar[d]\\
  \crG_{\ov{p}}N
  \ar[r]&\crG_{\ov{p}}X.
 }
 \end{equation}
With the existence of a stratified homotopy equivalence $\holinks(N,S)\simeq_{s} \holinks(X,S)$, the
juxtaposition of \eqref{equa:firstpush} and \eqref{equa:secondpush} is a homotopy pushout:
\begin{equation}\label{equa:lastpush}
\xymatrix{
 \crG_{\ov{p}}(\holinks(N,S))\ar[r]\ar[d]&\crG_{\ov{p}}(X\menos S)\ar[d]\\
\widetilde{P}_{D\ov{p}(S)}(\crG_{\ov{p}}\holinks(N,S))\ar[r]&\crG_{\ov{p}}X.
 }
 \end{equation}
 By induction and \remref{rem:pairscylinders}, we have a series of weak equivalences:\\
 $\cL_{\Phi}\holinks(X,S)\simeq |\crG_{\ov{p}}\holinks(X,S)|\simeq |\crG_{\ov{p}}\holinks(N,S)$,
 $|\crG_{\ov{p}}(X,S)|\simeq \cL_{\Phi}(X\menos S)$
 and\\
$|\widetilde{P}_{D\ov{p}(S)}(\crG_{\ov{p}}\holinks(N,S))|\simeq \widetilde{P}_{D\ov{p}(S)}(\cL_{\Phi}\holinks(X,S))$.
 Therefore, the diagram (\ref{equa:lastpush}) gives a homotopy pushout,
 \begin{equation}\label{equa:verylastpush}
\xymatrix{
 \cL_{\Phi}\holinks(X,S)\ar[r]\ar[d]&
 \cL_{\Phi}(X\menos S)\ar[d]\\
\widetilde{P}_{D\ov{p}(S)}(\cL_{\Phi}\holinks(X,S))\ar[r]&|\crG_{\ov{p}}X|,
 }
 \end{equation}
 and the result follows.
  \end{proof}

  In the particular case of a space with two strata, 
 \thmref{thm:gajer2strata} reduces to \thmref{prop:gajerneighbor}.

  \begin{example}\label{exam:examplesteenrod1}
  Let $F\to E\to B$ be a manifold bundle over manifold base.
 With a fibrewise conification, we get a fibration
  \begin{equation}\label{equa:exampleonestratum1}
\rc F\to X\to B.
\end{equation}
 This fibration admits a section, $x\in B\mapsto \tv_{x}\in X$, where $\tv_{x}$ is the apex of the fibre over $x$. 
We thus get an identification of the base $B$ as a closed subset of $X$ and we filter $X$ by
$\emptyset \subset X_{0}=B\subset X_{1}=X$.
The singular subset is $B$ and the link of $x\in B$ in $X$ is $F$. 
The Quinn presentation 
expresses $X$ as the homotopy pushout,
\begin{equation}\label{equa:twostrata}\xymatrix{
\holinks(E\times ]0,1],B)\ar[rr]^-{\ev_{1}}_-{\simeq}\ar@{->>}[d]_{\ev_{0}}&&
E\times ]0,1] \ar[d] \\
B\ar[rr]^-{\simeq}&&
X.
}\end{equation}
Thus the realisation of $\crG_{\ov{p}}X$ is the fibrewise $D\ov{p}(B)$-Postnikov localization of
$F\to E\to B$,
and the $\ov{p}$-intersection homotopy groups of \cite{CST10}, $\pi_{*}^{\ov{p}}X$, fits into the long exact sequence
$$\xymatrix{
\dots\ar[r]&
\pi_{j+1}B\ar[r]&
\pi_{j}P_{D\ov{p}(\tv)} F\ar[r]&
\pi_{j}^{\ov{p}}X\ar[r]&
\pi_{j}B\ar[r]&\dots
}$$
If we add hypotheses of nilpotency and finite type (as in \cite[Section 6]{CST2}), the fibration
$F\to E\to B$ admits a Sullivan model (\cite{MR0646078}),
$(\land Z,d)\to (\land Z\otimes \land W,D)\to (\land W, \ov{D})$.
The indice $\ell$ of a fibrewise Postnikov $\ell$-localization corresponds to the degree of the graded vector space $W$.
Thus, 
$(\land Z\otimes \land W^{\leq \ov{p}(B)},{D})$
is a Sullivan model of the Gajer space $\crG_{\ov{p}}X$.
In a future work, we will connect this construction with that of the perverse minimal model introduced in
\cite{CST1}.
\end{example}

\providecommand{\bysame}{\leavevmode\hbox to3em{\hrulefill}\thinspace}
\providecommand{\MR}{\relax\ifhmode\unskip\space\fi MR }
\providecommand{\MRhref}[2]{%
  \href{http://www.ams.org/mathscinet-getitem?mr=#1}{#2}
}
\providecommand{\href}[2]{#2}

\end{document}